\documentclass[reqno, 12pt]{amsart}
\usepackage{amsmath, amssymb, amsthm}
\usepackage[dvipdfmx]{hyperref, graphicx} 
\usepackage[all]{xy} 
\usepackage{arydshln} 
\usepackage{mathtools} 
\usepackage{bm} 

\newenvironment{parts}[0]{%
  \begin{list}{}%
    {\setlength{\itemindent}{0pt}
     \setlength{\labelwidth}{1.5\parindent}
     \setlength{\labelsep}{.5\parindent}
     \setlength{\leftmargin}{2\parindent}
     \setlength{\itemsep}{0pt}
     }%
   }%
  {\end{list}}
\newcommand{\Part}[1]{\item[\upshape#1]}

\makeatletter
\@addtoreset{equation}{section}
\makeatother
\mathtoolsset{showonlyrefs}

\def\house#1{{%
    \setbox0=\hbox{$#1$}
    \vrule height \dimexpr\ht0+1.4pt width .4pt depth \dp0\relax
    \vrule height \dimexpr\ht0+1.4pt width \dimexpr\wd0+2pt depth \dimexpr-\ht0-1pt\relax
    \llap{$#1$\kern1pt}
    \vrule height \dimexpr\ht0+1.4pt width .4pt depth \dp0\relax
}}

\allowdisplaybreaks[3] 

\DeclareMathOperator{\Nor}{Nor} 
\DeclareMathOperator{\Conj}{Conj} 
\DeclareMathOperator{\Ker}{Ker} 
\DeclareMathOperator{\GL}{GL} 
\DeclareMathOperator{\PGL}{PGL} 
\DeclareMathOperator{\rank}{rank} 
\DeclareMathOperator{\M}{M}
\DeclareMathOperator{\Gal}{Gal}

\newcommand{\Z}{\mathbb{Z}} 
\newcommand{\Q}{\mathbb{Q}} 
\newcommand{\QB}{\overline{\mathbb{Q}}} 
\newcommand{\R}{\mathbb{R}}

\newcommand{\Qtr}{\mathbb{Q}^{\rm tr}}

\newcommand{\h}{\mathfrak{h}} 
\newcommand{\hop}{h_{\rm op}} 
\newcommand{\hs}{h_{\rm s}}

\newcommand{\ie}{\textit{i.e.}}
\newcommand{\eg}{\textit{e.g.}} 
\newcommand{\f}{\frac} 
\newcommand{\g}{\gamma}

\renewcommand{\t}{\text} 
\renewcommand{\l}{\left} 
\renewcommand{\r}{\right}

\renewcommand{\Pr}{\mathbb{P}} 
\renewcommand{\c}{_{\rm conj}}

\newtheorem{Theorem}{Theorem}[section] 
\newtheorem{Proposition}[Theorem]{Proposition}
\newtheorem{Lemma}[Theorem]{Lemma}
\newtheorem{Corollary}[Theorem]{Corollary}
\newtheorem{Question}[Theorem]{Question}

\theoremstyle{definition}

\newtheorem{Remark}[Theorem]{Remark}
\newtheorem{Example}[Theorem]{Example}
\newtheorem*{acknowledgement}{Acknowledgements}
\theoremstyle{plain}

\title[Northcott numbers for generalized weighted heights]{Northcott numbers for generalized weighted Weil heights} 
\author[M. Okazaki, K. Sano]
{Masao Okazaki, Kaoru Sano}
\address[M. Okazaki]{National Institute of Technology (KOSEN), Oyama College, Tochigi, 323-0806, Japan.}
\email{m-okazaki@oyama-ct.ac.jp}
\address[K. Sano]{NTT Institute for Fundamental Mathematics, NTT Communication Science Laboratories, NTT Corporation, 2-4, Hikaridai, Seika-cho, Soraku-gun, Kyoto 619-0237, Japan}
\email{kaoru.sano@ntt.com}

\subjclass[2020]{Primary 11G50, Secondary 11C20, 15B33}
\keywords{weighted height, Northcott number, finite extensions, operator height, spectral height}

\begin{document}
\begin{abstract}
    We give a generalization of weighted Weil heights.
    These heights generalize both Weil's heights and Dobrowolski's height.
    We study Northcott numbers for our heights.
    Our results generalize the authors' former work on Vidaux and Videla's question about the Northcott number.
    As an application, we evaluate Northcott numbers for Talamanca's spectral height on matrices.
\end{abstract} 
\maketitle

\section{Introduction}\label{Intro} 

Let $S$ be a set and $\h\colon S\longrightarrow\R_{\geq 0}$ be a function. 
We set 
\[ 
    B(S,\h, C)=\l\{ s\in S \mid \h(s)<C \r\} 
\]
for each $C\in\R_{>0}$. 
We also set 
\[ 
    \Nor(S,\h)=\inf\l\{ C\in\R_{>0} \mid \#B(S,\h,C)=\infty \r\}. 
\]
The non-negative real number $\Nor(S,\h)$ is called the {\it $\h$-Northcott number of $S$},
introduced in \cite{PTW} and \cite{VV}. 
Here note that we regard the value $\inf\emptyset$ to be $\infty$. 
It is said that $S$ has the {\it $\h$-Northcott property} or $\h$-(N) for short
(resp. {\it $\h$-Bogomolov property} or $\h$-(B) for short)
if the equality $\Nor(S,\h)=\infty$ (resp. the inequality $\Nor(S\setminus Z(S,\h),\h)>0$) holds, where 
\[ 
    Z(S,\h)=\{ a\in S \mid \h(a)=0 \} 
\] 
(see \cite{BZ} and \cite{PTW}). 
Note that $\h$-(B) follows from $\h$-(N). 

In this paper, we fix an algebraic closure $\QB$ of $\Q$ and consider any algebraic extension of $\Q$ to be a subset of $\QB$. 
Let $h\colon\QB\longrightarrow\R_{\geq0}$ be the absolute (logarithmic) Weil height. 
We will recall the definition of $h$ in Section \ref{Weil}. 
In number theory, it is a basic problem to calculate the $h$-Northcott numbers. 
Vidaux and Videla gave the following. 

\begin{Question}[{\cite[Question 6]{VV}}]\label{VidauxVidela} 
    Which real numbers can be realized as $\Nor(L,h)$ for some field $L\subset\QB$? 
\end{Question} 

On the basis of the former works \cite{Wid} and \cite{PTW}, the authors generalized Question \ref{VidauxVidela} and answered to it in \cite{OS}. 

\begin{Theorem}[{\cite[Theorem 1.4]{OS}}]\label{SugoiTeiri} 
    Let $c$ be a positive real number and $\g$ a real number with $\g<1$. 
    Then we can construct a field $L\subset\QB$ satisfying the equality 
    \[
        \Nor(L,h_\g)=c, 
    \]
    where $h_\g(a)=\deg(a)^\g h(a)$ {\rm (}$a\in\QB$, $\deg(a)=[\Q(a):\Q]${\rm )} is the $\g$-weighted Weil height, introduced in \cite{PTW}. 
\end{Theorem} 

\begin{Remark} 
    The case $\gamma=0$ of Theorem \ref{SugoiTeiri} answers Question \ref{VidauxVidela}. 
    In \cite{Hul}, Hultberg applied this result to the finiteness of cycles of small height on projective varieties. 
\end{Remark} 

In this paper, we introduce $w$-weighted heights $h^w$ that further generalize the $\g$-weighted Weil heights. 
The purpose of this paper is to generalize Theorem \ref{SugoiTeiri} to our heights $h^w$. 
We say that a function $w\colon\Z_{>0}\longrightarrow\R_{>0}$ ia a \textit{weight} if it is eventually non-increasing or non-decreasing. 
For a weight $w$ and $a\in\QB$, we set 
\[ 
    h^w(a)=w(\deg(a))h(a). 
\]
We call $h^w$ \textit{the $w$-weighted Weil height}. 
We remark that the $\g$-weighted Weil height $h_\g$ coincides with $h^{w_\g}$, where $w_\g(d)=d^\g$. 
Another remarkable example is Dobrowolski's height. 
For the weight 
\[ 
    w_{\t{D}}(d)=\l(\f{\log^+(d)}{\log^+(\log(d))}\r)^3d, 
\]
Dobrowolski in \cite{Dob} proved that $\QB$ has $h^{w_{\t{D}}}$-(B), where $\log^+(\cdot)=\max\{1,\log(\cdot)\}$. 
A merit to introduce our generalized heights is the fact that for any infinite algebraic extension of $\Q$,
there exists a weight $w$ such that $\Nor(L\setminus\mu,h^w)$ is a positive real number,
where $\mu$ is the set of roots of unity in $\QB$ (see Corollary \ref{Infinite Boundary}). 

In addition, we address not only fields but also sets of algebraic numbers of bounded degree: 
For a field $L\subset\QB$ and $N\in\Z_{>0}$, we set 
\[ 
    L^{(N)}=\l\{ a\in\QB \mid [L(a):L]\leq N \r\}. 
\] 
Our motivations to focus on the sets $L^{(N)}$ are a number of topics regarding the Northcott property and the Bogomolov property
(\eg, \cite{ASZ}, \cite{AZ}, \cite{Dob}, \cite[Theorem 2.1]{DZ}, \cite[Conjecture 3.24]{Sil2}). 

Our main result is the following. 

\begin{Theorem}\label{implicit} 
    Let $c$ be a positive real number, $N$ a positive integer, and $w$ a weight with the following two conditions. 
    \begin{parts} 
    \Part{(1)}\hypertarget{condition(1)'} 
        $\lim_{d\rightarrow\infty}w(d)\log(d)/d=0$. 
    \Part{(2)}\hypertarget{condition(2)'} 
        $w(d)/d$ is eventually non-increasing. 
    \end{parts} 
    Then we can construct a field $L\subset\QB$ satisfying the equality 
    \[ 
        \Nor(L^{(N)},h^w)=c. 
    \] 
\end{Theorem} 

\begin{Remark} 
    Both conditions (\hyperlink{condition(1)'}{1}) and (\hyperlink{condition(2)'}{2}) in Theorem \ref{implicit}
    demand that a weight $w$ is smaller enough than the weight $w_1(d)=d$. 
    Here, we refer to Lehmer's conjecture, which asserts that $\QB$ has $h^{w_1}$-(B). 
\end{Remark} 

\begin{Remark} 
    If $N=1$ or the limit $\lim_{d\rightarrow\infty}w(d)$ converges to a positive real number,
    we can remove the condition (\hyperlink{condition(2)'}{2}) as shown in the proof of Theorem $\ref{explicit}$.  
\end{Remark}

As an application of Theorem \ref{implicit}, we study Northcott numbers of matrices. 
Let $\hs$ be the spectral height on $\M_n(\QB)$, introduced in \cite{Tal}. 
The height $\hs$ is obtained from Tate's limiting process
starting with the operator height $\hop$ (see Section \ref{s} for the definitions of $\hs$ and $\hop$). 
Hence, $\hs$ is an analogue of the canonical heights on abelian varieties. 
Now, we should refer to the following Talamanca's work. 

\begin{Theorem}[{\cite[Theorem 4.2]{Tal}}]\label{Talamanca} 
    For each $A\in{\rm M}_n(\QB)$,  we denote by $(\lambda_1^{(A)},\ldots,\lambda_n^{(A)})$ the $n$-tuple formed by the all eigenvalues of $A$. 
    Then, we have the equality 
    \[ 
        \hs(A)=h(\lambda_1^{(A)},\ldots,\lambda_n^{(A)}), 
    \] 
    where $h$ is the Weil height on $\Pr^{n-1}(\QB)\cup\{\bm{0}\}$ {\rm(}for the definition, see Section \ref{Weil}{\rm)}. 
\end{Theorem} 

Combining the product formula in Section \ref{Weil}, to study $\hs$-Northcott numbers,
we should consider $\hs$ to be a function on the set $\Conj(\M_n(\QB))$. 
Here, we set 
\[ 
    \Conj(\M_n(L))=\M_n(L)/\mathord{\sim}\c 
\] 
for each field $L\subset\QB$, where $A\mathbin{\sim}\c B$ ($A, B\in\M_n(L)$) means that
there exist $P\in\GL_n(\QB)$ and $c\in\QB^\times$ such that $B=c(P^{-1}AP)$. 
We denote by $[A]\c$ the class of $A\in\M_n(L)$ in $\Conj(\M_n(L))$. 
We also set 
\[ 
    \Conj(\GL_n(L))=\{ [A]\c \in\Conj(\M_n(L)) \mid A\in\GL_n(L) \}. 
\] 
We address $\Conj(\M_n(L))$ motivated by thinking about
the $\PGL_n(\QB)$-conjugacy class of rational self-maps on the projective space $\Pr^{n-1}(\QB)$.  
It is natural to ask about the existence of inequality relations among $\Nor(L,h)$, $\Nor(\Conj(\M_n(L)),\hs)$, and $\Nor(\Conj(\GL_n(L)),\hs)$. 
More precisely, we want to find inequalities
\begin{align} 
    &\hphantom{=} c_1\Nor(L,h) 
    \leq \Nor(\Conj(\M_n(L),\hs) \\
    &\leq \Nor(\Conj(\GL_n(L),\hs)
    \leq c_2\Nor(L,h) 
    \label{Norineq} 
\end{align}
for some $c_1, c_2\in\R_{>0}$. 
However, the first author proved the following proposition.

\begin{Proposition}[{\cite[Theorem 1]{Oka1}}]\label{Master Okazar} 
    We have the inequalities 
    \[
        \Nor(\Qtr, h)>\Nor(\Conj(\M_n(\Qtr)),\hs)=0, 
    \]
    where $\Qtr$ is the field of totally real numbers. 
\end{Proposition} 

\begin{Remark} 
    The fact that $\Qtr$ has $h$-(B) was first proven in \cite{Sch}. 
\end{Remark}  

Proposition \ref{Master Okazar} reveals that it is impossible to seek for inequalities \eqref{Norineq} for a general field $L\subset\QB$. 
Hence we will attempt to give an example of a field $L\subset\QB$ satisfying inequalities \eqref{Norineq}. 
As we did for $h$, we address this problem not only for $\hs$ but also the {\it $w$-weighted spectral height $\hs^w$}
(for the definition, see Section \ref{s}). 
In Section \ref{s}, for a weight $w$ with several conditions, we will give a field $L\subset\QB$ and $c_1,c_2\in\R_{>0}$ satisfying the inequalities 
\begin{align} 
    0 &< c_1\Nor(L,h^w) \leq \Nor(\Conj(\M_n(L),\hs^w) \\
    &\leq \Nor(\Conj(\GL_n(L),\hs^w)
    \leq c_2\Nor(L,h^w) 
    <\infty
    \label{wNorineq} 
\end{align} 
by using results on Theorem \ref{implicit}.

\subsection*{Notation} 
\begin{parts} 
    \Part{$\bullet$} 
        For a number field $K$, we denote by $\mathcal{M}_K$ the set of all places of $K$. 

    \Part{$\bullet$} 
        For a field $L\subset\QB$, we denote by $\Pr^n(L)$ the $n$-dimensional projective space over $L$. 
        We denote by $[\bm{a}]$ (resp. $[a_0:a_1:\cdots:a_n]$)
        the class of $\bm{a}\in L^{n+1}\setminus\{\bm{0}\}$ (resp. $(a_0,a_1,\ldots,a_n)\in L^{n+1}\setminus\{\bm{0}$\}) in $\Pr^n(L)$. 

    \Part{$\bullet$} 
        For a field $L\subset\QB$ and $\bm{a}=(a_0,\ldots,a_n)\in\QB^{n+1}\setminus\{\bm{0}\}$,
        we fix a non-zero entry $a_i$ and set 
        \[
            \deg_L(\bm{a})=[L(a_j/a_i \mid 0\leq j\leq n):L], 
        \]
        while we set $\deg_L(\bm{0})=1$. 
        We may consider $\deg_L$ to be a function on $\Pr^n(\QB)$. 
        We simply denote $\deg_\Q$ by $\deg$. 
    \Part{$\bullet$} 
        We denote by $\mu$ the set of roots of unity in $\QB$. 
        We also set $\mu_0=\{0\}\cup\mu$. 
\end{parts}

\subsection*{Outline of this paper} 
In Section \ref{Weil}, we quickly summarize definitions and basic properties of the Weil heights. 
In Section \ref{Weight section}, we give two fundamental properties of the $w$-weighted Weil heights. 
In Section \ref{Part1}, we prove Theorem \ref{implicit}. 
In Section \ref{s}, we introduce the $w$-weighted spectral heights $\hs^w$ on matrices and study inequalities \eqref{wNorineq}. 
Section \ref{op} is an appendix. 
We address the $\hop$-version of inequalities \eqref{wNorineq}.

\section{A quick review on Weil heights}\label{Weil}

In this section, we summarize the definitions and basic properties of the Weil heights (see \cite{HS} for details).
We will frequently use them in the latter sections without mentioning them.

Let $K$ be a number field.  
For each $v\in\mathcal{M}_K$, we fix the absolute value $|\cdot|_v\colon K\longrightarrow\R_{\geq0}$ such that the product formula 
\[ 
    \prod_{v\in\mathcal{M}_K}|a|_v^{[K_v:\Q_v]}=1 
\] 
holds for all $a\in K$, where $K_v$ (resp. $\Q_v$) is the completion of $K$ (resp. $\Q$) by $|\cdot|_v$. 
For each $\bm a=(a_0,\ldots,a_n)\in K^{n+1}\setminus\{\bm0\}$, we set 
\[ 
    h(\bm a)=\sum_{v\in\mathcal{M}_K}\f{[K_v:\Q_v]}{[K:\Q]}\log\l(\max_{0\leq i\leq n}\{|a_i|_v\}\r) 
\] 
and $h(\bm0)=0$. 
It is well-known that the value $h(\bm a)$ is independent of the choice of a number field $K$. 
Combining the product formula, we may consider $h$ to be a function on $\Pr^n(\QB)$. 
The function $h$ is called {\it the Weil height on $\Pr^n(\QB)$}. 

By abuse of notation, for each $a\in\QB$, we set 
\[
    h(a)=h(1,a). 
\]
The function $h$ is called {\it the Weil height on $\QB$}. 
The following are basic formulae of the Weil height on $\QB$. 

\begin{Lemma}\label{without mentioning it} 
    \ 
    \begin{parts} 
        \Part{(1)} 
            For each $a\in\QB$ and $d\in\Z_{>0}$, we have 
            \[
                h(a^{1/d})=\f{1}{d}h(a). 
            \]
        \Part{(2)} 
            For each $p, q\in\Z\setminus\{0\}$ with $\gcd(p,q)=1$, we have 
            \[
                h(p/q)=\log\l(\max\{|p|,|q|\}\r). 
            \] 
        \Part{(3)}
            For each $a, b\in\QB$, we have 
            \[
                h(ab)\leq h(a)+h(b). 
            \]
    \end{parts} 
\end{Lemma} 

\begin{proof} 
    The assertions follow from the definition of $h$. 
    For proofs, see, \eg, \cite[Section 1.5]{BG}. 
\end{proof} 

The following fundamental theorem for the Weil heights is known as Northcott's property.

\begin{Theorem}\label{North} 
    Let $L\subset\QB$ be a field with $h${\rm -(N)}. 
    Then, for any positive integer $N$, the set 
    \[
        \{ [\bm{a}]\in\Pr^n(\QB) \mid \deg_L(\bm{a})\leq N \}
    \]
    has $h${\rm -(N)}. 
    In particular, the set $L^{(N)}$ has $h${\rm -(N)}. 
\end{Theorem} 

\begin{proof} 
    See, \eg, \cite[Theorem 2.1]{DZ}. 
\end{proof} 

\begin{Remark} 
    Lemma \ref{without mentioning it} (2) and Theorem \ref{North} imply that any number field has $h^w$-(N) for any weight $w$. 
\end{Remark} 

We also give the following Kronecker's theorem. 

\begin{Theorem}\label{Kronecker} 
    The equality $Z(\QB,h)=\mu_0$ holds. 
\end{Theorem} 

\begin{proof} 
    See, \eg, \cite[Theorem 1.5.9]{BG}. 
\end{proof}

\section{Weights}\label{Weight section}

In this section, we give two fundamental properties of the $w$-weighted Weil heights. 
However, we will not use these properties in the latter sections. 

\begin{Proposition}\label{Boundary} 
    Let $w$ and $u$ be weights with $\lim_{d\rightarrow\infty}u(d)/w(d)=\infty$. 
    If a set $S\subset\QB$ has $h^w${\rm -(B)}, then the set $S\setminus\mu$ has $h^u${\rm -(N)}. 
\end{Proposition} 

\begin{proof} 
    Take any $C\in\R_{>0}$. 
    We claim that there are only finitely many $a\in S\setminus\mu_0$ satisfying the inequality $h^u(a)<C$. 
    Since $S$ has $h^w$-(B), there exists $D\in\R_{>0}$ independent of $a$ such that $h^w(a)\geq D$. 
    Hence, we have the inequality 
    \[
        \f{u(\deg(a))}{w(\deg(a))}=\f{h^u(a)}{h^w(a)}<\f{C}{D}. 
    \]
    Combining the assumption that $\lim_{d\rightarrow\infty}u(d)/w(d)=\infty$,
    we notice that there exists $M\in\Z_{>0}$ independent of $a$ such that 
    \begin{align}
        \deg(a)\leq M. 
        \label{w1} 
    \end{align} 
    Thus we have the inequality 
    \begin{align} 
        h(a) 
        =\f{h^u(a)}{u(\deg(a))} 
        <\f{C}{\min\l\{u(i) \mid 1\leq i\leq M\r\}}. 
        \label{w2} 
    \end{align} 
    By Theorem \ref{North}, the inequalities \eqref{w1} and \eqref{w2} imply that
    there are only finitely many choices for $a\in S\setminus\mu_0$. 
\end{proof} 

\begin{Remark} 
    Proposition \ref{Boundary} generalizes \cite[Proposition 2.3]{OS}. 
\end{Remark} 

\begin{Remark} 
    We did not use the fact that $w$ and $u$ are eventually non-decreasing or eventually non-increasing
    in the above proof of Proposition \ref{Boundary}. 
\end{Remark}

\begin{Theorem}\label{Existence of Boundary} 
    Let $S\subset\QB$ be a set with $\#(S\cap\mu)<\infty$.
    Then there exists a weight $w$ such that $\Nor(S,h^w)$ is a positive finite real number if and only if $\{ \deg(a) \mid a\in S \}$ is an infinite set. 
\end{Theorem} 

\begin{proof} 
    We may assume that $S\cap\mu_0=\emptyset$. 

    First, assume that there exists a weight $w$ such that $\Nor(S,h^w)$ is a positive real number. 
    This implies that the set $S$ does not have $h^w$-(N). 
    Hence, there exists an infinite subset $S'\subset S$ such that $\sup\{ h^w(a) \mid a\in S' \}<\infty$. 
    If $\{ \deg(a) \mid a\in S' \}$ is a finite set, we also have $\sup\{ h(a) \mid a\in S' \}<\infty$. 
    This contradicts Theorem \ref{North}. 

    Next, assume that $\{ \deg(a) \mid a\in S \}$ is an infinite set. 
    We divide the proof into the following three cases. 

    \begin{parts} 
        \Part{(1)} 
            \underline{$\Nor(S,h)=\infty$.} 

            \vspace{1mm} 

            Since $S$ has $h$-(N), for each $d\in\Z_{>0}$, we can take an element $a_d\in S$ satisfying that $\deg(a_d)\geq d$ and
            \[
                h(a_d)=\min\{ h(a) \mid a\in S,\, \deg(a)\geq d \}. 
            \]
            Note that $\{ a\in S \mid \deg(a)\geq d\}\neq \emptyset$ for all $d\in\Z_{>0}$ and in fact
            \begin{align}
                \#\{a_d \mid d\in\Z_{>0}\}=\infty
                \label{w2.5} 
            \end{align}
            by the assumption that $\{ \deg(a) \mid a\in S \}$ is an infinite set. 
            We also note that the sequence $(h(a_d))_{d\in\Z_{>0}}$ is non-decreasing. 
            Hence,
            \[
                w(d)=\f{1}{h(a_d)} \quad (d\in\Z_{>0}) 
            \]
            is a non-increasing weight. 
            We remark that 
            \begin{align} 
                w(\deg(a_d))=\f{1}{h(a_d)}
            \label{w3} 
            \end{align}
            holds for all $d\in\Z_{>0}$. 
            This is because $h(a_{\deg(a_d)})=h(a_d)$ holds: 
            \begin{parts}
                \Part{$\bullet$} 
                    By $\deg(a_d)\geq d$, we have $h(a_{\deg(a_d)})\geq h(a_d)$;  
                    \Part{$\bullet$}
                    Since $a_d$ is an element of $\{ a\in S \mid \deg(a)\geq \deg(a_d) \}$, we have $h(a_{\deg(a_d)})\leq h(a_d)$. 
            \end{parts}
            Now, we claim that $\Nor(S,h^w)=1$. 
            \begin{parts}
                \Part{$\bullet$} 
                    By \eqref{w3}, for all $d\in\Z_{>0}$, we have
                    \[
                        h^w(a_d)
                        = w(\deg(a_d))h(a_d) 
                        = \f{h(a_d)}{h(a_d)} 
                        = 1.
                    \]
                    By \eqref{w2.5}, the above equalities imply that $\Nor(S,h^w)\leq 1$ holds. 
                \Part{$\bullet$} 
                    For each $a\in S$, we set $d_1=\deg(a)$ and then observe that  
                    \[
                        h^w(a) 
                        = \f{h(a)}{h(a_{d_1})} 
                        \geq 1 \quad \t{by the definition of } a_{d_1}. 
                    \]
                    The above inequality implies that $\Nor(S,h^w)\geq 1$ also holds. 
            \end{parts} 
            Thus, we have $\Nor(S,h^w)=1$. 
            This means that $w$ is a weight function that we want.

        \Part{(2)}
            \underline{$\Nor(S,h)\in\R_{>0}$.} 

            \vspace{1mm} 

            As this case is trivial, set $w(d)=1$.

        \Part{(3)}
            \underline{$\Nor(S,h)=0$.} 

            \vspace{1mm}

            Let $d_0=\min\{\deg(a)\mid a\in S\}$. 
            By Theorem \ref{North}, for each $d\in\Z_{\geq d_0}$, we can take an element $a_d\in S$
            satisfying that $\deg(a_d)\leq d$ and 
            \[
                h(a_d)=\min\{ h(a) \mid a\in S,\, \deg(a)\leq d \}. 
            \]
            Since $S$ does not have $h$-(B), it is implied that $\lim_{d\rightarrow\infty}h(a_d)=0$
            (this is a special case of Lemma \ref{Nnumber}). 
            Hence we have 
            \begin{equation}\label{w3.5}
                \#\left\{a_d \mid d\in\Z_{>0}\right\} = \infty. 
            \end{equation}
            
            We also note that the sequence $(h(a_d))_{d\geq d_0}$ is non-increasing. 
            Hence,
            \begin{equation}
                w(d) =
                \begin{cases} 
                    1 & (d\in\Z_{>0},\, d<d_0) \\
                    1/h(a_d) & (d\in\Z_{\geq d_0}) 
                \end{cases} 
            \end{equation}
            is an eventually non-decreasing weight. 
            We remark that 
            \begin{align} 
                w(\deg(a_d))=\f{1}{h(a_d)}
                \label{w4}
            \end{align}
            holds for all $d\in\Z_{\geq d_0}$. 
            This is because $h(a_{\deg(a_d)})=h(a_d)$ holds: 
            \begin{parts}
                \Part{$\bullet$}
                    By $\deg(a_d)\leq d$, we have $h(a_{\deg(a_d)})\geq h(a_d)$; 
                \Part{$\bullet$} 
                    Since $a_d$ is an element of $\{ a\in S \mid \deg(a)\leq \deg(a_d) \}$,
                    we have $h(a_{\deg(a_d)})\leq h(a_d)$. 
            \end{parts}
            Now, we claim that $\Nor(S,h^w)=1$. 
            \begin{parts}
                \Part{$\bullet$}
                    By \eqref{w4}, for all $d\in\Z_{\geq d_0}$, we have 
                    \[
                        h^w(a_d)
                        = w(\deg(a_d))h(a_d) 
                        =\f{h(a_d)}{h(a_d)} 
                        = 1.
                    \]
                    By \eqref{w3.5}, the above equalities imply that $\Nor(S,h^w)\leq 1$ holds. 
                \Part{$\bullet$} 
                    For each $a\in S$, we set $d_2=\deg(a)$ and then observe that 
                    \[
                        h^w(a) 
                        = \f{h(a)}{h(a_{d_2})} 
                        \geq 1 \quad \t{by the definition of } a_{d_2}. 
                    \]
                    The above inequality implies that $\Nor(S,h^w)\geq 1$ also holds. 
            \end{parts}
            Thus, we have the equality $\Nor(S,h^w)=1$. 
            This means that $w$ is a weight function that we want. 
        \end{parts}
    \end{proof}

\begin{Corollary}\label{Infinite Boundary} 
    Let $L\subset\QB$ be an infinite extension of $\Q$. 
    Then, there exists a weight $w$ such that $\Nor(L\setminus\mu,h^w)$ is a positive real number. 
\end{Corollary}

\section{Northcott numbers for the generalized weighted heights}\label{Part1} 

We will prove Theorem \ref{implicit} in Subsection \ref{finite extensions} of this section. 

\subsection{Preparations from group theory}\label{Group} 

All statements in this section will be used to prove several properties of the tower of field extensions in Subsection \ref{Field} through the Galois theory. 
Let $d$ be a prime number. 
By abuse of notation, for each $a\in \Z$, we denote $a+d\Z\in\Z/d\Z$ by $a$. 
Let $G$ be a semi-direct product of $\Z/d\Z$ and $(\Z/d\Z)^\times$ defined by the following binary operation: 
\[ 
    (a_1,b_1)\cdot(a_2,b_2) = (a_1+a_2b_1,b_1b_2) 
\]
for $(a_1,b_1), (a_2,b_2)\in\Z/d\Z\times(\Z/d\Z)^\times$. 
We also let $\pi\colon G\longrightarrow (\Z/d\Z)^\times$ be the canonical projection. 

\begin{Lemma}\label{Sanosan2} 
    \ 
    \begin{parts} 
        \Part{(1)}\hypertarget{Sanosan2(1)}  
            The group $G$ has only one subgroup of order $d$.

        \Part{(2)}\hypertarget{Sanosan2(2)} 
            The group $G$ has exactly $d$ subgroups of order $d-1$.
            
        \Part{(3)}\hypertarget{Sanosan2(3)} 
            Any subgroup $H\subset G$ with $d\mid|H|$ contains the subgroup of $G$ of order $d$.

        \Part{(4)}\hypertarget{Sanosan2(4)}
            Any subgroup $H\subset G$ with $d\nmid|H|$ is contained in a subgroup of $G$ of order $d-1$.
    \end{parts} 
\end{Lemma} 

\begin{proof} 
    \
    \begin{parts} 
        \Part{(\hyperlink{Sanosan2(1)}{1})} 
            We prove that 
            \[
                G_1=\l\{ (a,1)\in G \mid a\in\Z/d\Z \r\}\subset G
            \]
            is the only subgroup of order $d$. 
            Note that any Sylow $d$-subgroup of $G$ is a conjugate of $G_1$. 
            On the other hand, the group $G_1$ is a normal subgroup of $G$ by definition. 
            Thus, $G_1$ is the only subgroup of $G$ of order $d$. 

        \Part{(\hyperlink{Sanosan2(2)}{2})} 
            Fix a generator $b_0\in(\Z/d\Z)^\times$. 
            Let $H\subset G$ be a subgroup with $|H|=d-1$. 
            Since $\Ker(\pi|_H)$ is a subgroup of both $H$ and $\Ker(\pi)$, the value $|\Ker(\pi|_H)|$ divides both $|H|=d-1$ and $|\Ker(\pi)|=d$. 
            Thus, the equality $|\Ker(\pi|_H)|=1$ holds and $\pi|_H$ is an isomorphism. 
            Hence, the group $H$ is generated by an element of $\pi^{-1}(b_0)$. 
            Therefore, $H$ is of the form
            \begin{align} 
                H_a=\langle(a,b_0)\rangle\subset G 
                \label{n11} 
            \end{align} 
            for some $a\in\Z/d\Z$. 
            Thus, it is sufficient to prove that
            the equality $|H_a|=d-1$ holds for all $a\in\Z/d\Z$ and
            that $H_{a_1}\neq H_{a_2}$ holds for distinct two elements $a_1,a_2\in\Z/d\Z$. 

            First, let $a\in\Z/d\Z$ and $s\in\Z\cap[1,d-1]$. 
            Noting that $b_0-1\neq0$, we have the equalities 
            \begin{equation} 
                \label{n22}
                (a,b_0)^s
                = \l(a\sum_{i=1}^sb_0^{i-1}, b_0^s\r) 
                = \l(a\f{b_0^s-1}{b_0-1}, b_0^s\r). 
            \end{equation} 
            Since $|(\Z/d\Z)^\times|=d-1$ holds, the equalities \eqref{n22} yield that 
            \[ 
                (a,b_0)^{d-1}=\l(a\f{b_0^{d-1}-1}{b_0-1}, b_0^{d-1}\r)=(0,1). 
            \]
            Thus, $|H_a|$ divides $d-1$.
            On the other hand, we have $(a,b_0)^s\neq(a,b_0)^t$ for distinct $s,t\in\Z_{\leq d-1}$
            since $b_0^s\neq b_0^t$ holds.
            Therefore, the equality $|H_a|=d-1$ holds.

            Next, assume that the equality $H_{a_1}=H_{a_2}$ holds.
            Then, there exists an integer $s \in [1,d-1]$ such that $(a_1,b_0)=(a_2,b_0)^s$.
            By \eqref{n22}, we have the equalities
            \[
                (a_1,b_0)=(a_2,b_0)^s=\l(a_2\f{b_0^s-1}{b_0-1}, b_0^s\r).
            \]
            Thus, we get $s=1$.
            Therefore, we conclude that $a_1=a_2$.

        \Part{(\hyperlink{Sanosan2(3)}{3})} 
            Let $H\subset G$ be a subgroup with $d\mid |H|$. 
            Then, $H$ has a Sylow $d$-subgroup $H'$. 
            Since $H'$ is also a Sylow subgroup of $G$, the equality $H'=G_1$ holds by (\hyperlink{Sanosan2(1)}{1}), where $G_1$ is as in (\hyperlink{Sanosan2(1)}{1}).

        \Part{(\hyperlink{Sanosan2(4)}{4})} 
            Let $H\subset G$ be a subgroup with $d\nmid |H|$. 
            From the equality $|G|=d(d-1)$, we know that $|H|$ divides $d-1$.
            We set $m=(d-1)/|H|\in\Z_{>0}$. 
            By the similar discussion of (\hyperlink{Sanosan2(2)}{2}), we get the injectivity of $\pi|_H$.
            Hence, $H$ is cyclic and generated by an element $g_0\in\pi^{-1}(b_0^m)$,
            where $b_0$ is as in (\hyperlink{Sanosan2(2)}{2}). 
            We write $g_0=(a_0,b_0^m)$ ($a_0\in\Z/d\Z$). 
            If $|H|=1$, then the assertion immediately follows. 
            Thus we assume that $|H|\geq2$, \ie, $m\leq d-2$. 
            We set $g=\l(a_0(b_0-1)(b_0^m-1)^{-1},b_0\r)$. 
            By \eqref{n22}, we have $g_0=g^m$. 
            Hence, $H$ is contained in the subgroup $H_{a_0(b_0-1)(b_0^m-1)^{-1}}$ of order $d-1$ defined in \eqref{n11}. 
    \end{parts} 
\end{proof}

\subsection{Preparations from field theory}\label{Field} 

The following Proposition \ref{Sanosan} plays a crucial role in the proof of Theorem \ref{implicit}. 
The purpose of this section is to prove it. 

\begin{Proposition}\label{Sanosan} 
    Let $N$ be a positive integer and $d$ be a prime number with $d>N$. 
    Assume that a primitive $d$-th root of unity $\zeta_d\in\QB$ satisfies that $[K(\zeta_d):K]=d-1$. 
    Let $\alpha$ be an algebraic number whose minimal polynomial over $K$ is $x^d-\alpha^d\in K[x]$. 
    We set $M=K(\alpha)$. 
    For each $a\in M^{(N)}$, we also set $s=[M(a):M]$. 
    Then, if $s<[K(a):K]$, there exists a non-negative integer $k\leq d-1$ such that $K(\zeta_d^k\alpha)\subset K(a)$.
\end{Proposition}

Let the notation be as in Proposition \ref{Sanosan} and $\widetilde{M}$ the Galois closure of $M/K$. 
We also let $\sigma_{(a,b)}\in\Gal(\widetilde{M}/K)$ be an element defined by 
\[
    \begin{cases} 
        \sigma_{(a,b)}(\alpha)=\zeta_d^a\alpha, \\
        \sigma_{(a,b)}(\zeta_d)=\zeta_d^b. 
    \end{cases} 
\]
Now, since $d=[M:K]$ divides $[\widetilde{M}:K]=[\widetilde{M}:K(\zeta_d)][K(\zeta_d):K]=[\widetilde{M}:K(\zeta_d)](d-1)$ and $d$ is a prime,
$d$ must divide $[\widetilde{M}:K(\zeta_d)]$. 
Since $\widetilde{M}$ is generated by $\alpha$ over $K(\zeta_d)$,
the inequality $[\widetilde{M}:K(\zeta_d)]\leq d$ holds. 
Thus, we have 
\begin{align}
    d=[\widetilde{M}:K(\zeta_d)].
    \label{d} 
\end{align} 
Therefore, the map 
\[
    G\longrightarrow\Gal(\widetilde{M}/K); \, 
    (a,b)\longmapsto \sigma_{(a,b)} 
\]
is a group isomorphism, where $G$ is the group defined in Section \ref{Group}.
Thus, we get the following assertion by Lemma \ref{Sanosan2} and Galois theory. 

\begin{Lemma}\label{Sanosan3} 
    Let $\widetilde{M}/K$ be the field extension in Proposition \ref{Sanosan}. 
    \begin{parts} 
        \Part{(1)}\hypertarget{Sanosan3(1)} 
            The field $K(\zeta_d)$ is the only intermediate field of $\widetilde{M}/K$ of degree $d-1$ over $K$. 
        
        \Part{(2)}\hypertarget{Sanosan3(2)} 
            The fields $K(\alpha), K(\zeta_d\alpha), \ldots, K(\zeta_d^{d-1}\alpha)$ are all intermediate fields of $\widetilde{M}/K$ of degree $d$ over $K$. 
        
        \Part{(3)}\hypertarget{Sanosan3(3)}
            Any intermediate field $M'$ of $\widetilde{M}/K$ with $d\nmid[M':K]$ is contained in $K(\zeta_d)$. 
        
        \Part{(4)}\hypertarget{Sanosan3(4)} 
            Any intermediate field $M'$ of $\widetilde{M}/K$ with $d\mid[M':K]$ contains $K(\zeta_d^k\alpha)$ for some integer $k\in [0,d-1]$.
    \end{parts} 
\end{Lemma} 

To prove Proposition \ref{Sanosan}, we also employ the following lemma. 

\begin{Lemma}\label{Sanosan1} 
    Let the notation be as in Proposition \ref{Sanosan} and $\widetilde{M}$ the Galois closure of $M/K$. 
    Then, we have the following assertions. 
    \begin{parts} 
        \Part{(1)}\hypertarget{Sanosan1(2)} 
            $d\mid[K(a):K]$. 
        \Part{(2)}\hypertarget{Sanosan1(3)} 
            $d\mid[\widetilde{M}\cap K(a):K]$. 
    \end{parts} 
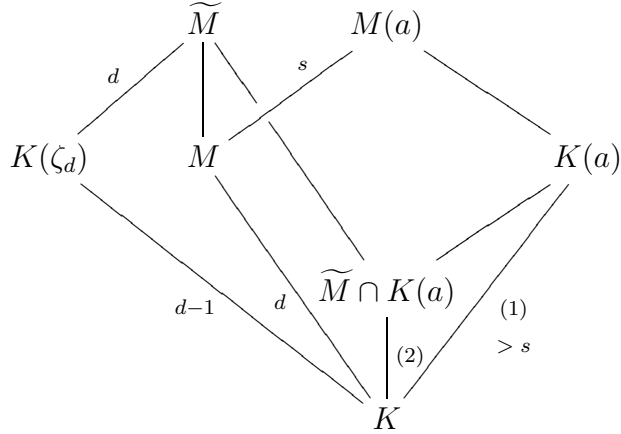
    \[
        \xymatrix@=30pt{
        & \widetilde{M} & M(a) &  \\
        K(\zeta_d) \ar@{-}[ru]^{d} & M \ar@{-}[u] \ar@{-}[ru]^(.6)s & & K(a) \ar@{-}[lu] \\ 
        & & \widetilde{M}\cap K(a) \ar@{-}[luu]|(.675)\hole \ar@{-}[ru] & \\
        & & K \ar@{-}[lluu]^{d-1} \ar@{-}[luu]^d \ar@{-}[u]_{(2)} \ar@{-}[ruu]_{\begin{matrix} \t{\tiny{\rm (1)}} \\ \t{\tiny{$>s$}} \end{matrix}} & \\
        }
    \] 
\end{Lemma} 

\begin{proof} 
    \
    \begin{parts} 
        \Part{(\hyperlink{Sanosan1(2)}{1})}
            Note that $[K(a):K]$ divides $[M(a):K]=[M(a):M][M:K]=sd$.
            By the assumption that $[K(a):K]>s$, the prime number $d$ must divide $[K(a):K]$. 
            
        \Part{(\hyperlink{Sanosan1(3)}{2})}
            We prove the assertion by contradiction. 
            Suppose that $d\nmid[\widetilde{M}\cap K(a):K]$ holds. 
            Under the assumption, we claim that $[K(a)(\alpha):K(a)]=d$ holds. 
            It is sufficient to show that $x^d-\alpha^d$ is irreducible over $K(a)$. 
            Consider the irreducible decomposition of $x^d-\alpha^d$ over $K(a)$. 
            Clearly all coefficients of the irreducible factors of $x^d-\alpha^d$ are in $\widetilde{M}$, so are in $\widetilde{M}\cap K(a)$. 
            By the assumption that $d\nmid[\widetilde{M}\cap K(a):K]$ and Lemma \ref{Sanosan3} (\hyperlink{Sanosan3(3)}{3}), the field $\widetilde{M}\cap K(a)$ is contained in $K(\zeta_d)$. 
            Thus, the irreducible decomposition of  $x^d-\alpha^d$ over $K(a)$ is also a decomposition of it over $K(\zeta_d)$.
            Now, note that $[K(\zeta_d)(\alpha):K(\zeta_d)]=d$ holds by \eqref{d}. 
            Hence, the irreducible decomposition of $x^d-\alpha^d$ over $K(a)$ is trivial. 
            Thus, we have $[K(a)(\alpha):K(a)]=d$. 
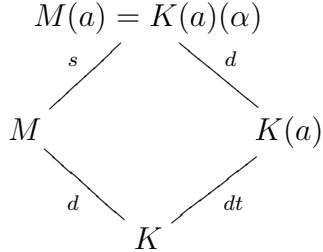
            \[
                \xymatrix@C=-10pt{
                & M(a)=K(a)(\alpha) & \\
                M \ar@{-}[ru]^s & & K(a) \ar@{-}[lu]_d \\
                & K \ar@{-}[lu]^d \ar@{-}[ru]_{dt} & \\
                }
                \] 
            By ({\hyperlink{Sanosan1(2)}{1}}), we have $[K(a):K]=dt$ for some $t\in\Z_{>0}$. 
            Hence we have the equalities 
            \[
                sd=[M(a):M][M:K]=[M(a):K(a)][K(a):K]=d^2t. 
            \]
            This contradicts the assumption that $s\leq N<d$.
    \end{parts}
\end{proof}

\begin{proof}[Proof of Proposition \ref{Sanosan}] 
    By Lemma \ref{Sanosan1}, the prime number $d$ divides $[\widetilde{M}\cap K(a):K]$. 
    Hence, $\widetilde{M}\cap K(a)$ contains $K(\zeta_d^k\alpha)$ for some integer $k\in [0,d-1]$ by Lemma \ref{Sanosan3} (\hyperlink{Sanosan3(4)}{4}). 
    By the inclusion $\widetilde{M}\cap K(a)\subset K(a)$, the assertion follows. 
\end{proof}

\subsection{Proof of Theorem \ref{implicit}}\label{finite extensions} 

In this subsection, we prove Theorem \ref{implicit}. 
More precisely, we prove the following explicit one. 

\begin{Theorem}\label{explicit} 
    Let $c$ be a positive real number, $N$ be a positive integer, and $w$ be a weight with the following two conditions. 
    \begin{parts} 
        \Part{(1)}\hypertarget{condition(1)} 
            $\lim_{d\rightarrow\infty}w(d)\log(d)/d=0$. 
        \Part{(2)}\hypertarget{condition(2)} 
            $w(d)/d$ is eventually non-increasing. 
    \end{parts}
    Fix $d_0\in\Z_{>0}$ such that the following two conditions hold. 
    \begin{parts} 
        \Part{($d_0$-1)}\hypertarget{d01} 
            $w(d)$ is non-increasing or non-decreasing on $\Z_{\geq d_0}$. 
        \Part{($d_0$-2)}\hypertarget{d02} 
            $w(d)/d$ is non-increasing on $\Z_{\geq d_0}$.
    \end{parts} 
    We also let $(d_i)_{i\in\Z_{>0}}$, $(p_i)_{i\in\Z_{>0}}$,
    and $(q_i)_{i\in\Z_{>0}}$ be strictly increasing sequences of prime numbers satisfying the inequalities 
    \begin{align} 
        & d_1\geq d_0,
        \label{n-1} \\ 
        & \exp\l(\f{Nd_1}{w(ND_1)}c\r)\geq 2, 
        \label{n0} \\
        & 4\exp\l(\f{Nd_i}{w(ND_i)}c\r) \leq \exp\l(\f{Nd_{i+1}}{w(ND_{i+1})}c\r), 
        \label{n1} \\
        & \exp\l(\f{Nd_i}{w(ND_i)}c\r) \leq p_i \leq 2\exp\l(\f{Nd_i}{w(ND_i)}c\r), 
        \label{n2} \\
        \t{and } \ 
        & p_i< q_i <2p_i 
        \label{n3} 
    \end{align}  
    for all $i\in\Z_{>0}$, where 
    \[
        D_i=
        \begin{cases} 
            d_i & (\t{if } \lim_{d\rightarrow\infty}w(d)>0), \\
            \prod_{j=1}^id_j & (\t{if } \lim_{d\rightarrow\infty}w(d)=0). 
        \end{cases} 
    \] 
    If $\lim_{d\rightarrow\infty}w(d)=0$, we furthermore assume that 
    \begin{align} 
        \lim_{i\rightarrow\infty}\f{iw(ND_i)}{w(ND_{i-1})}=0. 
        \label{n4} 
    \end{align} 
    We set $L=\Q((p_i/q_i)^{1/d_i} \mid i\in\Z_{>0})$. 
    Then, the equality $\Nor(L^{(N)},h^w)=c$ holds. 
\end{Theorem} 

\begin{Remark} 
    We can take a sequences $(d_i)_{i\in\Z_{>0}}$ satisfying \eqref{n0} and \eqref{n3} by the condition (1),
    that implies that $\lim_{d\rightarrow\infty}d/w(d)=\infty$. 
    The inequalities \eqref{n1}, \eqref{n2}, and \eqref{n3} imply that
    $\{p_i\ |\ I \in \Z_{\geq 0}\} \cap \{q_i\ |\ \Z_{\geq 0}\} = \emptyset$.  
    We can also take sequences $(p_i)_{i\in\Z_{>0}}$ and $(q_i)_{i\in\Z_{>0}}$
    satisfying \eqref{n1} and \eqref{n2} by \eqref{n0} and Bertrand-Chebyshev's theorem. 
    If $\lim_{d\rightarrow\infty}w(d)=0$,
    then we can always take a sequence $(d_i)_{i\in\Z_{>0}}$
    satisfying \eqref{n4} by letting,
    for example, a prime $d_i$ be so large that 
    \[ 
        w(ND_i) < \f{w(ND_{i-1})}{i^2}. 
    \] 
\end{Remark}

The proof of Theorem \ref{explicit} is further based on the following three lemmata. 

\begin{Lemma}\label{Disc} 
    Let $K'/M'/K$ be extensions of number fields. 
    Then 
    \[ 
        D_{K'/K}=N_{M'/K}(D_{K'/M'})D_{M'/K}^{[K':M']}. 
    \] 
\end{Lemma} 

\begin{proof} 
    See, \eg, \cite[p.202, (2.10)]{Neu}. 
\end{proof}

\begin{Lemma}\label{Siliq} 
    Let $K$ be a number field. 
    Assume that $a\in \QB$ satisfies that $[K(a):K]>1$. 
    We set $M=K(a)$ and $m=[M:K]$. 
    Then 
    \[
        h(a)\geq \f{1}{2(m-1)}\l(\f{\log(N_{K/\Q}(D_{M/K}))}{m[K:\Q]}-\log(m)\r) 
    \] 
    holds, where $N_{K/\Q}$ is the usual norm and $D_{M/K}$ is the relative discriminant ideal of the extension $M/K$. 
\end{Lemma} 

\begin{proof} 
    See \cite[Theorem 2]{Sil1}. 
\end{proof}

\begin{Lemma}\label{Nnumber} 
    Let $S\subset\QB$ be a subset and $w$ be a weight. 
    Assume that $S_0\subsetneq S_1\subsetneq S_2\subsetneq\cdots$ is an ascending chain of non-empty subsets of $S$ satisfying the following two conditions. 
    \begin{parts}
        \Part{(1)}
            $S_i$ has $h^w$-{\rm(N)} for all $i\in\Z_{\geq0}$. 
        
        \Part{(2)}
            $S=\bigcup_{i\in\Z_{\geq0}}S_i$.
    \end{parts} 
    Then we have 
    \[
        \Nor(A,h^w)=\liminf_{i\rightarrow\infty}\inf\l(S_i\setminus S_{i-1}\r). 
    \]
\end{Lemma}

\begin{proof} 
    See \cite[Lemma 6]{PTW}. 
\end{proof}

\begin{proof}[Proof of Theorem \ref{explicit}] 
    We set $\alpha_i=(p_i/q_i)^{1/d_i}$, $K_0=\Q$, and $K_i=K_{i-1}(\alpha_i)$ for each $i\in\Z_{>0}$. 
    First, we apply Lemma \ref{Nnumber} to $L^{(N)}=\bigcup_{i\in\Z_{\geq0}}K_i^{(N)}$. 
    Take any $i\in\Z_{>0}$ with $d_i>N$ and $a\in K_i^{(N)}\setminus K_{i-1}^{(N)}$. 
    By Proposition \ref{Sanosan}, there exists an integer $k\in [0,d_i-1]$ such that $K_{i-1}(\zeta_{d_i}^k\alpha_i)\subset K_{i-1}(a)$. 
    We set $s=[K_{i-1}(a):K_{i-1}(\zeta_{d_i}^k\alpha_i)]$. 
    Note that the inequalities 
    \begin{align} 
        sd_i\leq \deg(a) \leq s\prod_{j=1}^id_j 
        \label{pr1}
    \end{align} 
    hold by the equalities 
    \[ 
        sd_i=[K_{i-1}(a):K_{i-1}(\zeta_{d_i}^k\alpha_i)][K_{i-1}(\zeta_{d_i}^k\alpha_i):K_{i-1}]=[K_{i-1}(a):K_{i-1}] 
    \] 
    and 
    \begin{align*}
        s\prod_{j=1}^id_j& =[K_{i-1}(a):K_{i-1}(\zeta_{d_i}^k\alpha_i)][K_{i-1}(\zeta_{d_i}^k\alpha_i):K_{i-1}][K_{i-1}:\Q] \\
        & =[K_{i-1}(a):\Q].
    \end{align*} 
    We also have 
    \begin{align} 
        s 
        &= \f{[K_{i-1}(a):K_{i-1}(\zeta_{d_i}^k\alpha_i)][K_{i-1}(\zeta_{d_i}^k\alpha_i):K_{i-1}]}{[K_{i-1}(\zeta_{d_i}^k\alpha_i):K_{i-1}]} 
        = \f{[K_{i-1}(a):K_{i-1}]}{d_i} \\
        &\leq \f{[K_i(a):K_{i-1}]}{d_i} 
        = \f{[K_i(a):K_i][K_i:K_{i-1}]}{d_i} 
        = [K_i(a):K_i] \\
        &\leq N. 
        \label{pr0}
    \end{align} 
    Applying Lemma \ref{Disc} to the extensions 
    \[ 
        \begin{matrix} 
            K_{i-1}(a) & & /K_{i-1} & /\Q, \\
            K_{i-1}(a) & /K_{i-1}(\zeta_{d_i}^k\alpha_i) & & /\Q, & \t{and} \\
            & \ \, K_{i-1}(\zeta_{d_i}^k\alpha_i) & /K_{i-1} & /\Q, 
        \end{matrix}
    \]
    we have 
    \begin{align*} 
        D_{K_{i-1}(a)/\Q} 
        &= N_{K_{i-1}/\Q}(D_{K_{i-1}(a)/K_{i-1}})D_{K_{i-1}/\Q}^{[K_{i-1}(a):K_{i-1}]}, \\
        D_{K_{i-1}(a)/\Q} 
        &= N_{K_{i-1}(\zeta_{d_i}^k\alpha_i)/\Q}(D_{K_{i-1}(a)/K_{i-1}(\zeta_{d_i}^k\alpha_i)})D_{K_{i-1}(\zeta_{d_i}^k\alpha_i)/\Q}^s, \t{ and} \\
        D_{K_{i-1}(\zeta_{d_i}^k\alpha_i)/\Q}
        &= N_{K_{i-1}/\Q}(D_{K_{i-1}(\zeta_{d_i}^k\alpha_i)/K_{i-1}})D_{K_{i-1}/\Q}^{[K_{i-1}(\zeta_{d_i}^k\alpha_i):K_{i-1}]} 
    \end{align*} 
    Since $p_i$ and $q_i$ does not ramify in $K_{i-1}$ and ramify totally in $\Q(\zeta_{d_i}^k\alpha_i)$,
    we have 
    \[ 
        p_i, q_i\nmid D_{K_{i-1}/\Q} \ \t{ and } \
        (p_iq_i)^{[K_{i-1}:\Q](d_i-1)}\mid N_{K_{i-1}/\Q}(D_{K_{i-1}(\zeta_{d_i}^k\alpha_i)/K_{i-1}}) 
    \] 
    by \cite[Lemma 3.2]{Oka3}. 
    Thus, we get
    \[ 
        (p_iq_i)^{[K_{i-1}:\Q](d_i-1)s}\mid N_{K_{i-1}/\Q}(D_{K_{i-1}(a)/K_{i-1}}). 
    \] 
    By \eqref{n2}, we conclude that 
    \begin{align} 
        2[K_{i-1}:\Q](d_i-1)s\log(p_i)
        < \log(N_{K_{i-1}/\Q}(D_{K_{i-1}(a)/K_{i-1}})). 
        \label{pr2} 
    \end{align} 
    Applying Lemma \ref{Siliq} to the extension $K_{i-1}(a)/K_{i-1}$, we get the following inequalities 
    \begin{align} 
        h(a) 
        \geq&~ \f{1}{2(sd_i-1)}\l(\f{\log(N_{K_{i-1}/\Q}(D_{K_{i-1}(a)/K_{i-1}}))}{sd_i[K_{i-1}:\Q]}-\log(sd_i)\r) \\
        >&~ \f{1}{2(sd_i-1)}\l(\f{2(d_i-1)\log(p_i)}{d_i}-\log(sd_i)\r) \\ 
        \geq&~ \f{1}{2(sd_i-1)}\l(\f{2N(d_i-1)}{w(ND_i)}c-\log(sd_i)\r), 
        \label{pr3}
    \end{align} 
    where the second and the third inequalities follow from \eqref{pr2} and \eqref{n1}, respectively. 
    If $\lim_{d\rightarrow\infty}w(d)=\infty\t{ or }0$, then we have 
    \begin{align*} 
        &~ h^w(a) \\
        >&~ \f{w(sD_i)}{2(sd_i-1)}\l(\f{2N(d_i-1)}{w(ND_i)}c-\log(sd_i)\r) && \t{by (\hyperlink{d01}{$d_0$-1}), \eqref{n-1}, \eqref{pr1}, and \eqref{pr3}} \\
        =&~ \f{N(d_i-1)}{sd_i-1} \f{s}{N} \f{w(sD_i)}{sD_i} \f{ND_i}{w(ND_i)} c -\f{w(sD_i)\log(sd_i)}{2(sd_i-1)} \\
        \geq&~ \f{N(d_i-1)}{sd_i-1} \f{s}{N} c -\f{w(sD_i)\log(sd_i)}{2(sd_i-1)} && \t{by (\hyperlink{d02}{$d_0$-2}),  \eqref{n-1}, and \eqref{pr0}} \\
        \rightarrow&~ c \quad (\t{as $i\rightarrow\infty$}) && \t{by (\hyperlink{condition(1)}{1}).}
    \end{align*}
    On the other hand, if $\lim_{d\rightarrow \infty}w(d)$ is positive and finite, then we have 
    \begin{align*} 
        h^w(a) 
        >&~ \f{w(\deg(a))}{w(ND_i)} \f{N(d_i-1)}{sd_i-1} c -\f{w(\deg(a))\log(sd_i)}{2(sd_i-1)} && \t{by \eqref{pr3}} \\
        \rightarrow&~ \f{N}{s}c \quad (\t{as $i\rightarrow\infty$}) && \t{by (\hyperlink{condition(1)}{1}) and \eqref{pr1}} \\
        \geq&~ c && \t{by \eqref{pr0}.} 
    \end{align*} 
    Therefore, by Lemma \ref{Nnumber}, we have the inequality 
    \begin{align} 
        \Nor(L^{(N)},h^w) 
        \geq 
        c
        \label{prr}
    \end{align}
    in any cases.
    
    Next, if $\lim_{d\rightarrow\infty}w(d)>0$, then the inequalities 
    \begin{align*} 
        h^w(\alpha_i^{1/N}) 
        =&~ \f{w(Nd_i)}{Nd_i} \log(q_i) && \t{by \eqref{n2}} \\
        <&~ \f{w(Nd_i)}{Nd_i} \l(\f{Nd_i}{w(Nd_i)}c+\log(4)\r) && \t{by \eqref{n1} and \eqref{n2}} \\
        =&~ c+\f{w(Nd_i)}{Nd_i}\log(4) \\
        \rightarrow&~ c \quad (\t{as $i\rightarrow\infty$}) && \t{by (\hyperlink{condition(1)}{1})}
    \end{align*}
    hold. 
    On the other hand, if $\lim_{d\rightarrow\infty}w(d)=0$, then the inequalities 
    \begin{align*} 
        &~ h^w\l(\l(\prod_{j=1}^i\alpha_j\r)^{1/N}\r) \\
        =&~ \f{w(ND_i)}{N} h\l(\prod_{j=1}^i\alpha_j\r) \\ 
        \leq&~ \f{w(ND_i)}{N} \sum_{j=1}^ih(\alpha_j) \\ 
        =&~ \f{w(ND_i)}{N} \sum_{j=1}^i\f{\log(q_j)}{d_j} && \t{by \eqref{n2}} \\
        <&~ \f{w(ND_i)}{N} \l(\sum_{j=1}^i\f{Nc}{w(ND_j)}+i\log(4)\r) && \t{by \eqref{n1} and \eqref{n2}} \\ 
        \leq&~ c +\f{iw(ND_i)}{w(ND_{i-1})}c+iw(ND_i)\f{\log(4)}{N} && \t{by (\hyperlink{d01}{$d_0$-1}) and \eqref{n-1}} \\
        \rightarrow&~ c \quad (\t{as $i\rightarrow\infty$}) && \t{by \eqref{n4}} 
    \end{align*} 
    hold. 
    Therefore, we have the inequality 
    \begin{align} 
        \Nor(L^{(N)},h^w) 
        \leq 
        c 
        \label{prl} 
    \end{align} 
    in any cases. 
    
    The inequalities \eqref{prr} and \eqref{prl} complete the proof. 
\end{proof}

Then, we get the following stronger result, which will be used in the proof of Proposition \ref{spectral}. 

\begin{Theorem}\label{stronger} 
    Let $w$, $c$, and $L$ be as in Theorem \ref{explicit} of the case $N=1$. 
    We further assume that $w$ satisfies the condition
    \begin{parts} 
        \Part{(3)}\hypertarget{condition(3)} 
            For all $M\in\Z_{>0}$, the limit 
            \[
                l_w(M)=\lim_{d\rightarrow\infty}\f{w(Md)}{w(d)}
            \]
            exists and is positive.
    \end{parts}
    Then
    \[ 
        \Nor(L^{(M)},h^w)=\f{l_w(M)}{M}c 
    \] 
    holds for all $M\in\Z_{>0}$. 
\end{Theorem} 

\begin{proof} 
    First, we see that $\Nor(L^{(M)},h^w)\geq(l_w(M)/M)c$. 
    Let $K_i$ be as in the proof of Theorem \ref{explicit}. 
    Take any $i\in\Z_{>0}$ with $d_i>M$ and $a\in K_i^{(M)}\setminus K_{i-1}^{(M)}$. 
    We also let $s$ be as that in the proof of Theorem \ref{explicit}. 
    Then by the same discussion in the proof of Theorem \ref{explicit}, if $\lim_{d\rightarrow\infty}w(d)=\infty\t{ or }0$, we have 
    \begin{align*} 
        h^w(a) 
        >&~ \f{w(sD_i)}{2(sd_i-1)}\l(\f{2(d_i-1)}{w(D_i)}c-\log(sd_i)\r) \\
        =&~ \f{d_i-1}{sd_i-1} \f{w(sD_i)}{w(D_i)} c -\f{w(sD_i)\log(sd_i)}{2(sd_i-1)} \\ 
        \rightarrow&~ \f{l_w(s)}{s}c \quad (\t{as } i\rightarrow\infty) \\
        \geq&~ \f{l_w(M)}{M}c. 
        \end{align*} 
        If $\lim_{d\rightarrow\infty}w(d)$ is positive and finite, we have 
        \begin{align*} 
        h^w(a) 
        >&~ \f{w(\deg(a))}{w(D_i)} \f{d_i-1}{sd_i-1} c -\f{w(\deg(a))\log(sd_i)}{2(sd_i-1)} \\
        \rightarrow&~ \f{1}{s}c \quad (\t{as } i\rightarrow\infty) \\
        \geq&~ \f{1}{M}c \\
        =&~ \f{l_w(M)}{M}c. 
    \end{align*} 
    
    Next, we see that $\Nor(L^{(M)},h^w)\leq(l_w(M)/M)c$. 
    If $\lim_{d\rightarrow\infty}w(d)>0$, we have 
    \begin{align*} 
        h^w(\alpha_i^{1/M}) 
        <&~ \f{w(Md_i)}{Md_i} \l(\f{d_i}{w(d_i)}c+\log(4)\r) \\
        =&~ \f{1}{M}\f{w(Md_i)}{w(d_i)}c+\f{w(Md_i)}{Md_i}\log(4) \\
        \rightarrow&~ \f{l(M)}{M}c. \quad (\t{as } i\rightarrow\infty) 
    \end{align*} 
    If $\lim_{d\rightarrow\infty}w(d)=0$, we have 
    \begin{align*} 
        &~ h^w\l(\l(\prod_{j=1}^i\alpha_j\r)^{1/M}\r) \\
        <&~ \f{w(MD_i)}{M} \l(\sum_{j=1}^i\f{c}{w(D_j)}+i\log(4)\r) \\ 
        \leq&~ \f{1}{M}\f{w(MD_i)}{w(D_i)}c +\f{iw(MD_i)}{w(D_{i-1})}c+iw(MD_i)\f{\log(4)}{M} \\
        =&~ \f{1}{M}\f{w(MD_i)}{w(D_i)}c +\f{iw(D_i)}{w(D_{i-1})}\f{w(MD_i)}{w(D_i)}c+iw(MD_i)\f{\log(4)}{M} \\
        \rightarrow&~ \f{l(M)}{M}c. \quad (\t{as } i\rightarrow\infty) 
    \end{align*} 
\end{proof}

\begin{Example} 
    For $\g\in\R_{<1}$ and $w(d)=d^\g/\log(d)$, $d^\g$, or $d^\g\log(d)$, the field $L$ in Theorem \ref{stronger} coincides with that dealt in \cite[Theorem 4.1]{OS}. 
    For all $M\in\Z_{>0}$ and each of three weights $w$, we observe that $l_w(M)=M^\g$.  
    Therefore we have 
    \[ 
        \Nor(L^{(M)},h^w)=\f{c}{M^{1-\g}}. 
    \] 
    \end{Example} 
    
\begin{Example} 
    For a weight $w(d)=d^{-d^2}$, the field $L$ in Theorem \ref{explicit} of the case $N=1$ is an example of those in \cite[Proposition 5.5]{OS}. 
    However, we observe that $l_w(M)=0$ for all $M\in\Z_{\geq2}$. 
\end{Example}

\section{Northcott numbers for the weighted spectral heights}\label{s} 

In this section, we apply Theorem \ref{stronger} to study the inequalities \eqref{Norineq}. 
First, we quickly recall the spectral height $\hs$. 
For a number field $K$ and each $\bm{a}=(a_0,\ldots,a_n)\in K^{n+1}\setminus\{\bm{0}\}$, we set 
\begin{align*} 
    h_2(\bm{a})=
    &~ \sum_{\substack{v\in\M_K\\ v:\t{archimedean}}}\f{[K_v:\Q_v]}{[K:\Q]}\log\l(\max_{0\leq i\leq n}\{|a_i|_v\}\r) \\
    &~ +\sum_{\substack{v\in\M_K\\ v:\t{non-archimedean}}}\f{[K_v:\Q_v]}{[K:\Q]}\log\l(\sqrt{\sum_{i=0}^n|a_i|_v^2}\r) 
\end{align*} 
and $h_2(\bm{0})=0$. 
As the Weil height $h$, we may regard $h_2$ as a function on $\Pr^n(\QB)$. 
For each $A\in\M_n(\QB)$, we set 
\begin{align*} 
    \hop(A)& =\sup_{\bm{x}\in\QB^n}\l\{h_2(A\bm{x})-h_2(\bm{x})\r\} \ \t{ and} \\ 
    \hs(A)& =\lim_{k\rightarrow\infty}\f{\hop(A^k)}{k}. 
\end{align*} 
For each $A=(a_{ij})\in\M_n(\QB)\setminus\{O\}$, we fix a non-zero entry $a_{ij}$ and set 
\[ 
    \deg(A)=[\Q(a_{kl}/a_{ij} \mid 1\leq k,l\leq n):\Q], 
\] 
while we set $\deg(O)=1$. 
Here, we remark that, for example, 
\[ 
    A=\begin{pmatrix}1&1\\-1&1\end{pmatrix} 
    \t{ and }
    B=\begin{pmatrix}1+\sqrt{-1}&0\\0&1-\sqrt{-1}\end{pmatrix} 
\]
satisfy that $A\sim\c B$ and $\deg(A)=1<2=\deg(B)$. 
Hence we also set 
\[
    \deg([A]\c)=\min_{B\in [A]\c}\{\deg(B)\} 
\]
for $A\in\M_n(\QB)$.

\begin{Lemma}\label{deg vs deg} 
    For all $A\in{\rm M}_n(\QB)$, it holds that 
    \[
        n!\deg([A]\c) \geq \deg(\lambda_1^{(A)},\ldots,\lambda_n^{(A)}). 
    \]
\end{Lemma} 

\begin{proof} 
    Take any $A=(a_{ij})\in\M_n(\QB)\setminus\{O\}$. 
    We may assume that $a_{ij}=1$ for some integers $i,j\in [1,n]$. 
    Thus we have 
    \begin{align} 
        \deg(A)=[\Q(a_{kl} \mid 1\leq k,l \leq n):\Q]. 
        \label{7.4} 
    \end{align} 
    Let $X^n+b_1X^{n-1}+\cdots+b_{n-1}X+b_n\in \QB[X]$ be the characteristic polynomial of $A$. 
    We know that $b_m\in\Q(a_{kl} \mid 1\leq k,l \leq n)$ for all $m\in [1,n]$. 
    Then we have 
    \begin{align} 
        [\Q(b_m \mid 1\leq m\leq n]):\Q] \leq [\Q(a_{kl} \mid 1\leq k,l \leq n):\Q].
        \label{7.5}
    \end{align} 
    Now, note that $b_k\in\Q(\lambda_1^{(A)},\ldots,\lambda_n^{(A)})$ holds for all $0\leq k\leq n-1$. 
    Thus, we conclude that 
    \begin{align*} 
        &~\deg(\lambda_1^{(A)},\ldots,\lambda_n^{(A)}) \\ 
        \leq&~ [\Q(\lambda_m^{(A)} \mid 1\leq m\leq n):\Q] \\ 
        =&~ [\Q(b_m,\lambda_m^{(A)} \mid 1\leq m\leq n):\Q] \\ 
        =&~ [\Q(b_m,\lambda_m^{(A)} \mid 1\leq m\leq n):\Q(b_m \mid 1\leq m\leq n)] \\
        &~ \times[\Q(b_m \mid 1\leq m\leq n):\Q] \\ 
        \leq&~ [\Q(\lambda_m^{(A)} \mid 1\leq m\leq n):\Q][\Q(b_m \mid 1\leq m\leq n):\Q]\\
        \leq&~ n! \deg(A) \t{\hspace{12em} by \eqref{7.4} and \eqref{7.5}}. 
    \end{align*} 
    These complete the proof. 
\end{proof}

Now, we define our {\it weighted spectral heights $\hs^w$}. 
For a weight $w$ and $[A]\c\in\Conj(\M_n(\QB))$, we set 
\[ 
    \hs^w([A]\c)=w(\deg([A]\c))\hs(A). 
\] 

The purpose of this section is to prove the following. 

\begin{Proposition}\label{spectral} 
    Let $w$ and $L$ be as in Theorem \ref{stronger}. 
    We furthermore assume that $w$ is non-decreasing. 
    Then they hold that 
    \begin{align}
        &\hphantom{=} \f{l_w(n^2)}{n^2n!}\Nor(L,h^w)
        \leq \Nor(\Conj({\rm M}_n(L),\hs^w) \\
        & \leq \Nor(\Conj(\GL_n(L)),\hs^w) 
        \leq \f{1}{n-1}\Nor(L,h^w). 
    \end{align} 
\end{Proposition}

\begin{proof}
    First, we prove the left inequality. 
    Take any $C\in\R_{>0}$ with 
    \[
        C < \f{l_w(n^2)}{n^2n!}\Nor(L,h^w). 
    \]
    It is sufficient to prove that the set $B(\Conj(\M_n(L)), \hs^w, C)$ is finite. 
    Take any element $[A]\c\in B(\Conj(\M_n(L)),\hs^w,C)$. 
    Considering the Jordan form, there are only finitely many $[B]\c\in\Conj(\M_n(\QB))$
    such that $(\lambda_1^{(B)},\ldots,\lambda_n^{(B)})=\bm{0}$. 
    Therefore, we may assume $\lambda^{(A)}_1\neq0$.
    Fix $M\in\Z_{>0}$ such that $w(d)/d$ is non-increasing on $\Z_{>M}$. 
    By Lemma \ref{deg vs deg} and Theorem \ref{North},
    there are only finitely many $[B]\c\in B(\Conj(\M_n(L)), \hs^w, C)$ such that $\deg([B]\c)\leq M$. 
    Therefore, we may assume that $\deg([A]\c)>M$. 
    Thus, we have 
    \[ 
        \f{w(\deg([A]\c))}{\deg([A]\c)} 
        \geq \f{w(n!\deg([A]\c))}{n!\deg([A]\c)}. 
    \] 
    Hence, the inequality
    \[ 
        w(\deg([A]\c))
        \geq \f{1}{n!}w(n!\deg([A]\c)) 
    \]
    holds. 
    Therefore, we have 
    \begin{align*} 
        &~\f{l_w(n^2)}{n^2}\Nor(L,h^w) \\
        >&~ n!C \\
        >&~ n!h^w([A]\c) \\ 
        =&~ n!w(\deg([A]\c))h(\lambda_1^{(A)},\ldots,\lambda_n^{(A)}) && \t{by Theorem \ref{Talamanca}} \\
        \geq&~ w(n!\deg([A]\c))h(\lambda_1^{(A)},\ldots,\lambda_n^{(A)}) \\
        \geq&~ w(\deg(\lambda_1^{(A)},\ldots,\lambda_n^{(A)}))h(\lambda_1^{(A)},\ldots,\lambda_n^{(A)}) 
        && \t{by Lemma \ref{deg vs deg}} \\
        \geq&~ w(\deg(\lambda_i^{(A)}/\lambda_1^{(A)}))h(\lambda_i^{(A)}/\lambda_1^{(A)}) \\
        =&~ h^w(\lambda_i^{(A)}/\lambda_1^{(A)}) 
    \end{align*} 
    for all integers $i$ in $[2,n]$. 
    Since $\lambda_i^{(A)}/\lambda_1^{(A)}\in L^{(n^2)}$ holds, Theorem \ref{explicit} implies that there are only finitely many choices for the value of $\lambda_i^{(A)}/\lambda_1^{(A)}$. 
    Thus, we conclude that there are only finitely many possibilities for $[A]\c$.

    Next, we prove the right inequality. 
    For each $i\in\Z_{>0}$, let $d_i$ and $\alpha_i$ be as in Theorem \ref{explicit} and Theorem \ref{stronger}. 
    We set 
    \[ 
        A_i=
        \begin{cases} 
            \begin{pmatrix} 
                1&0 \\
                0&\alpha_i \\
            \end{pmatrix} 
            \in \M_2(L) 
            & (n=2), 
            \\
            \left( 
            \begin{array}{c:c} 
                1 & \huge{O} \\ 
                \hdashline 
                \huge{O} & {\begin{matrix}
                    0 & 0 & 0 & \cdots & 0 & \alpha_i \\
                    1 & 0 & 0 & \cdots & 0 & 0 \\
                    0 & 1 & 0 & \cdots & 0 & 0 \\
                    0 & 0 & 1 & \cdots & 0 & 0 \\
                    \vdots & \vdots & \vdots & \ddots & \vdots & \vdots \\
                    0 & 0 & 0 & \cdots & 1 & 0
                \end{matrix}} \\ 
            \end{array}
            \right) 
            \in\M_n(L) 
            & (n>2). 
        \end{cases} 
    \] 
    Whether $n=2$ or not, they hold that $\deg(A_i)=d_i$ and 
    \begin{align} 
        \l(\lambda_1^{(A_i)},\ldots,\lambda_n^{(A_i)}\r) 
        = \l(1, \alpha_i^{\f{1}{n-1}}, \alpha_i^{\f{1}{n-1}}\zeta_{n-1}, \ldots, \alpha_i^{\f{1}{n-1}}\zeta_{n-1}^{n-1} \r), 
        \label{eigen of A} 
    \end{align} 
    where $\zeta_{n-1}$ is a primitive $(n-1)$-th root of unity. 
    Here, note that 
    \begin{align} 
        h\l(1, \alpha_i^{\f{1}{n-1}}, \alpha_i^{\f{1}{n-1}}\zeta_{n-1}, \ldots, \alpha_i^{\f{1}{n-1}}\zeta_{n-1}^{n-1} \r) 
        =h\l(\alpha_i^{\f{1}{n-1}}\r) 
        \label{rootmulti} 
    \end{align} 
    holds since we have $|a\zeta|_v=|a|_v$ for all $a\in\QB$, $\zeta\in\mu$, and $v\in\mathcal{M}_{\Q(\zeta,a)}$. 
    Therefore, we observe that 
    \begin{align} 
        &~\hs^w([A_i]\c) \\
        \leq&~ w(d_i)\hs([A_i]\c) \\
        =&~ w(d_i)h\l(\lambda_1^{(A_i)},\ldots,\lambda_n^{(A_i)}\r) && \t{by Theorem \ref{Talamanca}} \\ 
        =&~ w(d_i)h\l(\alpha_i^{\f{1}{n-1}}\r) && \t{by \eqref{eigen of A} and \eqref{rootmulti}} \\ 
        =&~ \f{1}{n-1}h^w(\alpha_i) \\
        \rightarrow&~ \f{1}{n-1}\Nor(L,h^w) \quad (\t{as }i\rightarrow\infty) && \t{by the proof of Theorem \ref{stronger}.} 
    \end{align} 
    The inequalities above imply the right inequality. 
\end{proof}

\section{Appendix: Northcott numbers for weighted operator heights}\label{op} 

In Section \ref{s}, we studied relations between the $h^w$-Northcott number of a field $L$ and
that of the matrices with the entries in $L$. 
In this additional section, we address the same topic for $\hop$. 
For each weight $w$ and $A\in\M_n(\QB)$, we set 
\[ 
    \hop^w(A)=w(\deg(A))\hop(A). 
\] 
We may consider $\hop^w(A)$ to be a function on $\Pr(\M_n(\QB))$. 
Here for each field $L\subset\QB$, we set 
\[ 
    \Pr(\M_n(L))=\M_n(L)/\mathord{\sim}, 
\] 
where $A\mathbin{\sim}B$ ($A, B\in\M_n(L)$) means that there exists $c\in\QB^\times$ such that $B=cA$. 
We denote by $[A]$ the class of $A\in\M_n(L)$ in $\Pr(\M_n(L))$. 
We also set 
\begin{align} 
    \Pr(\M_n(L))_{\geq2}&=\{ [A]\in\Pr(\M_n(L)) \mid \rank(A)\geq2 \} \ \t{ and} \\
    \Pr(\GL_n(L))&=\{ [A] \in\Pr(\M_n(L)) \mid A\in\GL_n(L) \}. 
\end{align}

\begin{Proposition}\label{opNorth} 
    For a non-decreasing weight $w$ and a field $L\subset\QB$, they hold that 
    \begin{align} 
        \f{1}{2}\Nor(L,h^w) 
        \leq &\Nor(\Pr({\rm M}_n(L))_{\geq2},\hop^w) \\
        \leq &\Nor(\Pr(\GL_n(L)),\hop^w)
        \leq \Nor(L,h^w). 
    \end{align} 
\end{Proposition}

\begin{Remark} 
    We dealt with $\Pr({\rm M}_n(L))_{\geq2}$ in Proposition \ref{opNorth} since it is known that
    there are infinitely many $[A]\in\Pr(\M_n(\Q))$ of rank $1$
    such that $\hop(A)=0$ (see \cite[Theorem 1.1 (2)]{Oka2}) or $\hop(A)=\log(2)/2$ (see \cite[p.103]{Tal}), respectively. 
\end{Remark}

For each weight $w$ and $[\bm{a}]\in\Pr^n(\QB)$, we set 
\[ 
h^w(\bm{a})=w(\deg(\bm{a}))h(\bm{a}) \\
\ \t{ and } \ 
h_2^w(\bm{a})=w(\deg(\bm{a}))h_2(\bm{a}). 
\] 

Our proof of Proposition A.1 is based on the following Lemma \ref{diagonalop}, \ref{column vec ineq}, and \ref{vec North}. 

\begin{Lemma}\label{diagonalop} 
    Let $w$ be a weight and $a_1, a_2,\ldots,a_n\in\QB$. 
    Then it holds that
    \[
        \hop^w \l(
        \l(\begin{array}{c c}
            {\begin{array}{cc}
                a_1 & \\
                    & a_2
            \end{array}} & \huge{O} \\ 
          \huge{O} &
          {\begin{array}{c c}
                \ddots & \\
                 & a_n \end{array}}
        \end{array}\r)
        \r)
        = h^w(a_1, a_2, \ldots,a_n).
    \]
\end{Lemma} 

\begin{proof} 
    The assertion immediately follows from \cite[Theorem 3.3 (a)]{Tal} and well-known explicit formulae for operator norms (See, \eg, \cite[(3.1) and (3.2)]{Oka2}). 
\end{proof}

\begin{Lemma}\label{column vec ineq} 
    Let $w$ be a non-decreasing weight. 
    For each matrix $A\in{\rm M}_n(\QB)$ and integer $j\in[1,n]$,
    we denote by $\bm{a}_j$ the $j$-th column vector of $A$. 
    Then 
    \[ 
        h_2^w(\bm{a}_j)\leq \hop^w(A) 
    \] 
    holds for all integers $j$ in $[1,n]$. 
\end{Lemma} 

\begin{proof} 
    Consider the vector 
    \[ 
        \bm{e}_j 
        ={}^t(0,\ldots,0,\hspace{-1ex}\underset{\underset{\tiny{j\mbox{-th}}}{\wedge}}{1}\hspace{-1ex},0,\ldots,0) \in\QB^n
    \]
    for each integer $j\in [1,n]$.
    Then we see that 
    \begin{align*} 
        \hop^w(A) 
        &\geq w(\deg(A))\l(h_2(A\bm{e}_j)-h_2(\bm{e}_j)\r) \\
        &= w(\deg(A))h_2(\bm{a}_j) \\
        &\geq w(\deg(\bm{a}_j))h_2(\bm{a}_j) && \t{since $w$ is non-decreasing} \\
        &= h_2^w(\bm{a}_j) 
    \end{align*} 
    for all integers $j$ in $[1,n]$.
\end{proof}

\begin{Lemma}\label{vec North} 
    Let $w$ be an eventually non-decreasing weight and $L\subset\QB$ be a field. 
    Then, it holds that 
    \[ 
    \Nor(\Pr^n(L),h_2^w) 
    \geq \Nor(L,h^w). 
    \] 
\end{Lemma} 

\begin{proof} 
    The lemma immediately follows from the fact that 
    \[ 
        h_2^w(\bm{a}) \geq h^w(a_j/a_i) 
    \] 
    holds for all $[\bm{a}]=[a_0:\cdots:a_n]\in\Pr^n(\QB)$ and integers $i,j$ in $[1,n]$. 
\end{proof}

\begin{proof}[Proof of Proposition \ref{opNorth}] 
    If $\Nor(L,\hop^w)=0$, then the assertion immediately follows from Lemma \ref{diagonalop}. 
    Assume that we have the inequality $\Nor(L,h^w)>0$.
    The inequality $\Nor(\Pr(\M_n(L))_{\geq2},\hop^w)\leq\Nor(L,h^w)$ immediately follows from Lemma \ref{diagonalop}. 
    We prove the inequality $\Nor(L,h^w)/2\leq\Nor(\Pr(\M_n(L))_{\geq2},\hop^w)$. 
    Take any constant $C\in\R_{>0}$ with $C<\Nor(L,h^w)/2$. 
    It is sufficient to show that the set $B(\Pr(\M_n(L))_{\geq2}, \hop^w, C)$ is finite. 
    Take any $[A]=[(\bm{a}_1,\ldots,\bm{a}_n)]\in B(\Pr(\M_n(L))_{\geq2},\hop^w,C)$. 
    By Lemma \ref{column vec ineq}, each vector $\bm{a}_j$ is an element of $B(L^n,h^2_\g,C)$ for all $j$. 
    Since the inequalities 
    \begin{align} 
        C
        < \Nor(L,h^w)/2 
        \leq \Nor(\Pr^{n-1}(L),h_2^w)/2 
        < \Nor(\Pr^{n-1}(L),h_2^w) 
        \label{8.1} 
    \end{align} 
    hold by Lemma \ref{vec North}, we can chose $\bm{x}_1,\ldots,\bm{x}_m\in L^n$ satisfying the following condition.  
    \begin{parts} 
        \Part{$\bullet$}
        For any $\bm{y}\in B(L^n,h_2^w,C)$, there only exists an integer $j(\bm{y})\in [1,m]$
        such that $\bm{y}$ is a non-zero scalar multiple of $\bm{x}_{j(\bm{y})}$. 
    \end{parts}
    For each $\bm{y}\in B(L^n,h_2^w,C)\setminus\{\bm{0}\}$,
    we let $\lambda(\bm{y})\in L^\times$ satisfy that $\bm{y}=\lambda(\bm{y})\bm{x}_{j(\bm{y})}$ 
    (we consider $\lambda(\bm{0})=1$). 
    We may assume $m\geq 3$ since the inequality $\rank(A)\geq2$ holds. 
    We divide the cases into whether there exists a column vector $\bm{a}_p$ that has at least two non-zero entries. 
    
    First, we assume that such a column vector $\bm{a}_p={}^t(a_{1p},\ldots,a_{np})$ exists.
    Fix non-zero entries $a_{qp}$ and $a_{rp}$ ($1\leq q<r\leq n$).
    Note that the equalities $a_{qp}=\lambda(\bm{a}_p)x_{qp}$ and $a_{rp}=\lambda(\bm{a}_p)x_{rp}$ hold,
    where $x_{qp}$ (resp. $x_{rp}$) is the $q$-th (resp. $r$-th) entry of $\bm{x}_{j(\bm{a}_p)}$. 
    We also fix any non-zero column vectors $\bm{a}_k$ and $\bm{a}_l$ satisfying that
    $1\leq k<l\leq n$ and $\bm{x}_{j(\bm{a}_k)} \neq \bm{x}_{j(\bm{a}_l)}$. 
    We note that such $k$ and $l$ exist since $\rank(A)\geq2$. 
    We remark that the inequality 
    \begin{align} 
        \deg(a_{qp}\bm{a}_k+a_{rp}\bm{a}_l)\leq\deg(A) 
        \label{8.2} 
    \end{align} 
    holds. 
    Now, we set 
    \[ 
        \bm{a}_p' 
        ={}^t(0,\ldots,0,\hspace{-0ex}\underset{\underset{\tiny{k\mbox{-th}}}{\wedge}}{a_{qp}}\hspace{-0ex}, 
        0,\ldots,0,\hspace{-0ex}\underset{\underset{\tiny{l\mbox{-th}}}{\wedge}}{a_{rp}}\hspace{-0ex},0,\ldots,0) 
        \in L^n. 
    \] 
    Then, we have the inequalities 
    \begin{align*} 
        &~ h_2^w\l(\bm{x}_{j(\bm{a}_k)}+\f{x_{rp}}{x_{qp}}\f{\lambda(\bm{a}_l)}{\lambda(\bm{a}_k)}\bm{x}_{j(\bm{a}_l)}\r) \\
        =&~ h_2^w(a_{qp}\bm{a}_k+a_{rp}\bm{a}_l) \\
        =&~ w(\deg(a_{qp}\bm{a}_k+a_{rp}\bm{a}_l))\l(h_2(A\bm{a}_p')-h_2(\bm{a}_p')+h_2(\bm{a}_p')\r) \\
        \leq&~ w(\deg(A))\l(h_2(A\bm{a}_p')-h_2(\bm{a}_p')\r)+w(\deg(A))h_2(\bm{a}_p) && \t{by \eqref{8.2}} \\
        \leq&~ w(\deg(A))\hop(A)+w(\deg(A))\hop(A) && \t{by Lemma \ref{column vec ineq}} \\
        <&~ 2C \\
        <&~ \Nor(\Pr^{n-1}(L),h_2^w) && \t{by \eqref{8.1}}. 
    \end{align*} 
    Since $\bm{x}_{j(\bm{a}_k)}$ and $\bm{x}_{j(\bm{a}_l)}$ are linearly independent,
    there are only finitely many possibilities for $\bm{x}_{j(\bm{a}_k)}+(x_{rp}/x_{qp})(\lambda(\bm{a}_l)/\lambda(\bm{a}_k))\bm{x}_{j(\bm{a}_l)}$ by Lemma \ref{vec North}. 
    Here, note that there are only finitely many choices for the value of $x_{rp}/x_{qp}$ ($\neq0$). 
    Thus, we know that there are also only finitely many choices for the value of $\lambda(\bm{a}_l)/\lambda(\bm{a}_k)$. 
    Thus, all $[A]\in B(\Pr(\M_n^2(L), \hop^w, C)$ of the case is a class of a matrix $A$ constructed by the following process. 
    \begin{parts}
        \Part{(1)} 
            Choose any integers $k \in [1,n]$ and $j_k\in [1,m]$.
            We set $\bm{a}_k=\bm{x}_{j_k}$.
        \Part{(2)} 
            Choose any integers $l\in [1,n]\setminus\{k\}$ and $j_l\in [1,m]$ such that $\bm{x}_{j_k}$ is linearly independent of $\bm{x}_{j_k}$. 
            We set $\bm{a}_l=\lambda_l\bm{x}_{j_l}$ for some $\lambda_l\in L^\times$ of finitely many choices.
        \Part{(3)} 
            For each integer $k'\in [1,n] \setminus\{k,l\}$, choose any integer $j_{k'}\in [1,m]$. 
            Set $\bm{a}_{k'}=\lambda_{k'}\bm{x}_{j_{k'}}$ for some $\lambda_{k'}\in L^\times$ of finitely many choices: 
            \[
            \xymatrix@R=2.5pt@C=8.5pt{ 
            &&& \t{\tiny{indep.}} \ar@/^5pt/@{-}[drr] \ar@/_5pt/@{-}[dll] &&& \\
            A= ( \ \cdots & \bm{a}_k \ar@/_18pt/[rrrr]_{\t{finitely many choices}} & \cdots & \bm{a}_l & \cdots & \bm{a}_{k'} & \cdots \ ) 
            } 
            \] 
            \[ 
            \xymatrix@R=2.5pt@C=8.5pt{ 
            &&& \t{\tiny{dep.}}  \ar@/^12.5pt/@{-}[ddrr] \ar@/_12.5pt/@{-}[ddll] &&& \\
            && \t{\tiny{indep.}} \ar@/^5pt/@{-}[dr] \ar@/_5pt/@{-}[dl] && \t{\tiny{indep.}} \ar@/^5pt/@{-}[dr] \ar@/_5pt/@{-}[dl] && \\
            A= ( \ \cdots & \bm{a}_k \ar@/_18pt/[rr]_{\begin{matrix}\t{\tiny{finitely many}}\\\t{\tiny{choices}}\end{matrix}} & \cdots & \bm{a}_l \ar@/_18pt/[rr]_{\begin{matrix}\t{\tiny{finitely many}}\\\t{\tiny{choices}}\end{matrix}} & \cdots & \bm{a}_{k'} & \cdots \ ) 
            } 
            \] 
    \end{parts} 
    Therefore, we conclude that there are only finitely many choices for $[A]$ of the case. 
    
    On the other hand, assume that any column vector of $A$ has at most one non-zero entry. 
    Fix any non-zero entries $a_{ij},a_{kl}\neq 0$ with $1\leq i\neq k\leq n$ and $1\leq j\neq l\leq n$. 
    We note that such non-zero entries exist since $\rank(A)\geq2$. 
    We set 
    \[ 
        \QB^n_{jl}=\l\{ {}^t(x_1,\ldots,x_n)\in\QB^n \mid x_m=0 \t{ for all } m\neq j,l \r\}. 
    \] 
    Then we have the inequalities 
    \begin{align*} 
        \Nor(L,h^w) 
        &> \f{\Nor(L,h^w)}{2} \\
        &> C \\
        &> \hop^w(A) \\
        &\geq w(\deg(A))\sup_{\bm{x}\in\QB^n_{jl}}\l\{h_2(A\bm{x})-h_2(\bm{x})\r\} \\
        &= w(\deg(A))\hop\l(\begin{pmatrix} a_{ij}&0 \\ 0&a_{kl} \end{pmatrix}\r) \\ 
        &\geq w\l(\deg\l(\begin{pmatrix} a_{ij}&0 \\ 0&a_{kl} \end{pmatrix}\r)\r) \hop\l(\begin{pmatrix} a_{ij}&0 \\ 0&a_{kl} \end{pmatrix}\r) \\ 
        &= \hop^w\l(\begin{pmatrix} a_{ij}&0 \\ 0&a_{kl} \end{pmatrix}\r) \\ 
        &= h^w(a_{ij}, a_{kl}) \t{\hspace{8em} by Lemma }\ref{diagonalop}. 
    \end{align*} 
    Thus, there are only finitely many choices for the value of $a_{ij}/a_{kl}\in L$. 
    Therefore, we also conclude that there are only finitely many choices for $[A]$ of the case. 
    
    These complete the proof. 
\end{proof} 

\begin{Remark} 
    The proof above is greatly based on the proof of \cite[Theorem 4.3 (a)]{Tal}. 
\end{Remark}

\begin{acknowledgement} 
    The authors would like to express their gratitude to Professor Fabien Pazuki for detecting certain notational inaccuracies and bringing reference \cite{Hul} to their attention. 
    The first author was supported until March 2022 by JST SPRING, Grant Number JPMJSP2136.
    The second author is supported by JSPS KAKENHI Grant Number JP20K14300.
\end{acknowledgement}

\end{document}